\documentclass[12pt, reqno]{amsart}
\usepackage{amsmath, amsthm, amscd, amsfonts, amssymb, graphicx, color}
\usepackage[bookmarksnumbered, colorlinks, plainpages]{hyperref}

\newtheorem{theorem}{Theorem}[section]
\newtheorem{lemma}[theorem]{Lemma}

\theoremstyle{definition}
\newtheorem{definition}[theorem]{Definition}

\theoremstyle{remark}
\newtheorem{remark}[theorem]{Remark}
\numberwithin{equation}{section}

\begin{document}
\setcounter{page}{1}

%-------------------------- Pleased do not change the following line-------------------------------------------
\noindent \textcolor[rgb]{0.99,0.00,0.00}{This is a submission to one of journals of TMRG: BJMA or AFA}\\[.5in]
%--------------------------------------------------------------------------------------------------------------

\title[$L^1$- convergence of greedy algorithm]{$L^1$- convergence of greedy algorithm by generalized Walsh system}

\author[Sergo A. Episkoposian ]{Sergo A. Episkoposian  (Yepiskoposyan)}

\address{ Faculty of Applied Mathematics , State Engeniering University of Armenia, Yerevan, Teryan st.105,  375049, Armenia.}

\email{\textcolor[rgb]{0.00,0.00,0.84}{sergoep@ysu.am}}

%\dedicatory{This paper is dedicated to Professor ABCD}

\subjclass[2010]{Primary 42A65; Secondary 42A20.}

\keywords{generalized Walsh system, monotonic coefficients, greedy algorithm.}

\date{Received: xxxxxx; Revised: yyyyyy; Accepted: zzzzzz.
\newline \indent $^{*}$ Corresponding author}

\begin{abstract}
In this paper we consider the generalized Walsh system and a problem
$L^1- convergence$ of  greedy algorithm of functions after changing
the values on small set .
\end{abstract} \maketitle

\section{Introduction and preliminaries}

\noindent Let $a$ denote a fixed integer, $a\geq 2$ and put $\omega_a=e^{2 \pi
i \over a}$. Now we will give the definitions of generalized Rademacher and Walsh
systems \cite{C}.

\begin{definition}\label{def-1.1}
 The Rademacher  system of order $a$ is defined
by
$$\varphi_0(x)=\omega_a^k\ \ if \ \ x \in \left[ {k\over a}, {k+1\over a}\right),\ \ k=0,1,...,a-1,\ \ x\in[0,1)$$
and for $n\geq 0$
$$\varphi_n(x+1)=\varphi_n(x)=\varphi_0(a^nx).$$
\end{definition}
\begin{definition}\label{def-1.2}
 The generalized Walsh system of order $a$ is
defined by
$$\psi_0(x)=1,$$
and if $n=\alpha_1a^{n_1}+...+\alpha_s a^{n_s}$ where $n_1>...>n_s,$
then
$$\psi_n(x)=\varphi_{n_1}^{\alpha_1}(x)\cdot...\cdot \varphi_{n_s}^{\alpha_s}(x) .$$
\end{definition}
 Let's denote the generalized Walsh system of order $a$ by $\Psi_{a}$.

Note that $\Psi_2$ is the classical Walsh system.

The basic properties of the generalized Walsh system of order $a$
are obtained by   H.E.Chrestenson, R. Pely, J. Fine, W. Young, C.
Vatari, N. Vilenkin and others (see \cite{C}- \cite{Y}).

In this paper we consider $L^1$- convergence of  greedy algorithm
with respect to $\Psi_a$ system. Now we present the definition of
greedy algorithm.

Let $X$ be a Banach space  with a norm $||\cdot||=||\cdot||_X$ and a
basis $\Phi=\{\phi_k\}_{k=1}^\infty$, $||\phi_k||_X=1$, $k=1,2,..$ .

For a function $f\in X$ we consider the expansion
$$f=\sum_{k=1}^\infty a_k(f)\phi_k \ \ .$$

\begin{definition}\label{def-1.3}
 Let an element $f\in X$ be given. Then the
$m$-th greedy approximant of the function $f$ with regard to the
basis $\Phi$ is given by
$$G_m(f,\phi)=\sum_{k\in\Lambda}a_k(f) \phi_k,$$
where $\Lambda \subset \{1,2,...\}$ is a set of cardinality $m$ such
that
 $$|a_n(f)|\geq |a_k(f)|, \ \  n\in \Lambda,\ \  k\notin
\Lambda. $$
\end{definition}
In particular we'll say that the greedy approximant of $f\in
L^p[0,1]$, $p\geq 0$ converges with regard to the
 $\Psi_a$ , if the sequence $G_m(x,f)$ converges to $f(t)$ in
$L^p$ norm. This new and very important direction invaded many
mathematician's attention (see \cite{D-T}-\cite{E-2}).

T.W. K$\ddot{o}$rner \cite{K} constructed an $L^2$ function (then a
continuous function) whose greedy algorithm with respect to
trigonometric systems diverges almost everywhere.

V.N.Temlyakov in \cite{T} constructed a function $f$ that belongs to
all $L^p$, $1\leq p<2$ (respectively $p>2$), whose greedy algorithm
concerning trigonometric systems divergence in measure (respectively
in $L^p$, $p>2$), e.i. the trigonometric system are not a
quasi-greedy basis for $L^p$ if $1<p<\infty$.

In \cite{G-N} R.Gribonval and M.Nielsen proved that for any $1<p<\infty$ there exits a
function $f(x)\in L^p[0,1)$ whose greedy algorithm with respect to
$\Psi_2$- classical Walsh system diverges in $L^p[0,1]$.
Moreover, similar result for $\Psi_a$ system follows from Corollary
2.3. (see \cite{G-N}). Note also that in  \cite{E-1} and \cite{E-2} this result was
proved  for $L^1[0,1]$.

The following question arises naturally: is it possible to change
the values of any function $f$ of class $L^1$ on small set, so that
a greedy algorithm of new modified function concerning $\Psi_a$
system converges in the $L^1$ norm?

The classical {\bf C}-property of Luzin is well-known, according to
which every measurable function can be converted into a continuous
one be changing it on a set of arbitrarily small measure. This
famous result of Luzin \cite{L} dates back to 1912.

 Note that Luzin's idea of modification of a function improving
its properties was substantially developed later on.

In 1939, Men'shov \cite{M1} proved the following fundamental theorem.

\textbf{Theorem (Men'shov's $C$-strong property).}\label{Men'shov'}
\textit{\ Let
$f(x)$ be an a.e. finite measurable function on $[0,2\pi]$. Then for
each $\varepsilon>0$ one can define a continuous function $g(x)$
coinciding with $f(x)$ on a subset $E$ of measure
$|E|>2\pi-\varepsilon$ such that its Fourier series with respect to
the trigonometric system converges uniformly on $[0,2\pi]$.}

Further interesting results in this direction were obtained by many
famous mathematicians (see for example \cite{M2}-\cite{Pr}).

Particulary in 1991 M. Grigorian obtain the following result \cite{G1}:

\textbf{Theorem ($L^1$-strong property).}\label{Grigorian}
\textit{\ For each
$\varepsilon>0$ there exits a measurable set $E\subset [0,2\pi]$ of
measure $|E|>2\pi-\varepsilon$ such that for any function $f(x)\in
L^1[0,2\pi]$ one can find a function $g(x)\in L^1[0,2\pi]$
coinciding with $f(x)$ on $E$ so that its Fourier series with
respect to the trigonometric system converges to $g(x)$ in the
metric of $L^1[0,2\pi]$.}

In this paper we prove the following:

\begin{theorem}\label{main-1.4}
For any $\varepsilon\in (0,1)$ and for any function $f\in L^1[0,1)$
there is a function $g\in L^1[0,1)$, with $mes\{x\in [0,1)\ ;\ g\neq f\}<\varepsilon$, such that the nonzero fourier coefficients by absolute values
monotonically decreasing.
\end{theorem}

\begin{theorem}\label{main-1.5}
For any $0<\varepsilon <1$ and each function $f\in L^1[0,1)$ one can
find a function $g\in L^1[0,1),\ mes\{x\in [0,1)\ ;\ g\neq
f\}<\varepsilon$, such that its fourier series by $\Psi_{a}$ system $
L^1$ convergence to $g(x)$ and the nonzero fourier coefficients
by absolute values monotonically decreasing, i.e. the greedy algorithm by $\Psi_a$ system $L^1$-convergence.
\end{theorem}

The Theorems 1.1 and 1.2 follows from next more general Theorem 1.3,
which in itself is interesting:

\begin{theorem}\label{main-1.6}
For any $0<\varepsilon <1$ there exists a measurable set $E\subset
[0,1)$ with $|E|>1-\varepsilon$ and a series by $\Psi_a$ system of
the form
$$
\sum_{i=1}^\infty c_i\psi _i(x), \ \ |c_i|\downarrow0
$$
such that for any function $f\in L^1[0,1)$ one can find a function
$g\in L^1[0,1)$,
$$g(x)=f(x);\ \ \hbox{if}\ \ x\in E
$$
and the series of the form
$$
\sum_{n=1}^\infty \delta_nc_n\psi_n(x),\ \ \hbox{where}\ \
\delta_n=0\ \ \hbox{or}\ \ 1,
$$
which convergence to $g(x)$ in $L^1[0,1)$ metric and
$$
\left|\left| \sum_{n=1}^m \delta_nc_n\psi_n(x \right|\right|_1\leq
12\cdot||f||_1,\ \ \forall m\geq 1.
$$
\end{theorem}
\begin{remark} \label{Rem1}
 Theorems 1.6 for classical Walsh system
$\Psi_2$ was proved by M. Grigorian \cite{G2}.
\end{remark}

\begin{remark}\label{Rem2}
 From Theorem 1.5 follows that generalized Walsh
system $\Psi_a$ has $L^1$-strong property.
\end{remark}
\par\par\bigskip

\section{Basic Lemmas}
\par\par\bigskip
First we present some properties of $\Psi_{a}$ system (see Definition 1.2).

{\bf Property 1.} Each  $n$th Rademacher function  has period
$1\over a^n$ and
\begin{equation}\label{2.1}
\varphi_n(x)=const \in \Omega_a=\{ 1, \omega_a,\omega_a^2,....,
\omega_a^{a-1} \},
\end{equation}
if  $ x\in \Delta_{n+1}^{(k)}= \left[ {k\over a^{n+1}}, {k+1\over
a^{n+1}}\right)$, $ k=0,...,a^{n+1}-1$, $ n=1,2,....$.

It is also easily verified, that
\begin{equation}\label{2.2}
\left( \varphi_n(x) \right)^k =\left( \varphi_n(x) \right)^m, \ \
\forall n,k\in \mathcal{N},  \textit{where}\ \  m=k\ (\textit{mod} \
a)
\end{equation}

{\bf Property 2.} It is clear, that for any integer $n$ the Walsh
function $\psi_n(x)$ consists of a finite product of Rademacher
functions and accepts values from $\Omega_a$.

{\bf Property 3.} Let $\omega_a=e^{2 \pi i \over a}$. Then for any natural
number $m$ we have
\begin{equation}\label{2.3}
\sum_{k=0}^{a-1}\omega_a^{k\cdot m}=
\begin{cases}a\ ,\  if \ m\equiv 0(\hbox {mod}\  a) , \cr \\
 \cr 0 ,\ \ if \ m\neq 0(\hbox {mod}\ \ a)\ . \cr
 \end{cases}
\end{equation}
{\bf Property 4.} The generalized Walsh system $\Psi_a$, $a\geq 2$
is a complete orthonormal system in $L^2[0,1)$
and basis in $L^p[0,1)$, $p>1$ \cite{Pa}).\\
{\bf Property 5.} From definition 2 we have
\begin{equation}\label{2.4}
\psi_i(x)\cdot\psi_j(a^sx)=\psi_{j\cdot a^s+i}(x)\ ,\ \hbox{where}\
0\leq i\ ,\ j<a^s ,
\end{equation}
and particulary
\begin{equation}\label{2.5}
\psi_{a^k+j}(x)=\varphi_k(x) \cdot \psi_j(x),\ \ if\ \ 0 \leq j \leq a^k-1.
\end{equation}

Now for any $m=1,2,...$ and $1\leq k\leq a^m$  we put
$\Delta_m^{(k)}=\left[{k-1\over a^m},{k \over a^m}\right)$ and
consider the following function
\begin{equation}\label{2.6}
I_m^{(k)}(x)=
\begin{cases}
1\ ,\ \hbox{if}\ x\in [0,1)\setminus \Delta_m^{(k)}\ ,\cr 1-a^m\ ,\
\hbox{if} \ x\in \Delta_m^{(k)} ,
\end{cases}
\end{equation}
and periodically extend these functions on $R^1$ with period 1.

By $\chi_E (x)$ we denote the characteristic function of the set
$E$, i.e.
\begin{equation}\label{2.7}
\chi_E(x)=
\begin{cases}
1\ ,\ \hbox{if}\ x\in E\ ,\ \cr 0\ ,\ \hbox{if}\ x\notin E\ .
\end{cases}
\end{equation}

Then, clearly
\begin{equation}\label{2.8}
I_m^{(k)}(x)=\psi_0(x)-a^m\cdot\chi_{\Delta_m^{(k)}}(x)\ ,
\end{equation}
and for the natural numbers $m\geq 1\ \hbox {and}\ 1\leq i\leq
a^m$

\begin{equation}\label{2.9}
a_i(\chi_{\Delta_m^{(k)}})=\int_0^1\chi_{\Delta_m^{(k)}}(x)\cdot{\overline
{\psi_i}(x)}dx= \mathcal{A} \cdot {1\over a^m} ,\ \  0\leq i<a^m .
\end{equation}

\begin{equation}\label{2.10}
b_i(I_m^{(k)})=\int_0^1I_m^{(k)}(x){\overline{\psi_i}}(x)dx=
\begin{cases}
0\ ,\ \hbox{if}\ i=0\ \hbox{and}\ i\geq a^k\ , \cr -\mathcal{A}\ ,\
\hbox{if}\ 1\leq i< a^k\,
\end{cases}
\end{equation}
where $\mathcal{A}=const\in \Omega_a$ and $|\mathcal{A}|=1$.

Hence
\begin{equation}\label{2.11}
\chi_{\Delta_m^{(k)}}(x)=\sum_{i=0}^{a^k -
1}b_i(\chi_{\Delta_m^{(k)}})\psi_i(x)\ ,
\end{equation}

\begin{equation}\label{2.12}
I_m^{(k)}(x)= \sum_{i=1}^{a^k - 1}a_i(I_m^{(k)})\psi_i(x)\ .
\end{equation}

\begin{lemma}\label{lem 2.1} For any numbers $\gamma \neq 0$, $N_0>1$,
$\varepsilon \in(0,1)$ and interval by order $a$
$\Delta=\Delta_m^{(k)}=[{k-1\over a^m},{k\over a^m}),\ \
i=1,...,a^m$ there exists a measurable set $E \subset \Delta$ and
a polynomial $P(x)$ by $\Psi_a$ system of the form
$$
 P(x) = \sum_{k=N_0}^N c_k\psi_k(x)
$$
which satisfy the conditions:
$$
\hbox{coefficients }\ \{c_k \}_{k=N_0}^N \ \ \hbox{equal}\ \ 0 \ \
\hbox{or}\ \ -\mathcal{K}\cdot \gamma\cdot |\Delta|, \leqno 1)
$$
where $\mathcal{K}=const\in \Omega_a$, $|\mathcal{K}|=1,$
$$
 |E|> (1- \varepsilon)\cdot |\Delta|,\leqno 2)
$$
$$
P(x)=
 \begin{cases}
\gamma,\ \ \hbox{if}\ \  x \in E ;\cr  0,\ \ \hbox{if}\ \  x
\notin \Delta.\cr
\end{cases}
\leqno 3)
$$
$$
{1\over 2}\cdot |\gamma| \cdot |\Delta|<\int_0^1|P(x)|dx<2\cdot
|\gamma| \cdot |\Delta|.\leqno 4)
$$
$$
\max_{N_0\leq m \leq N}\int_0^1\big|\sum_{k=N_0}^m
c_k\psi_k(x)\big|<a\cdot |\gamma| \cdot \sqrt{|\Delta|\over
\varepsilon}.\leqno 5)
$$
\end{lemma}

\begin{proof}
 We take a natural numbers $\nu_0$ � $s$
so that
\begin{equation}\label{2.13}
\nu_0=\left[\log_a{\frac{1}{\varepsilon}}\right]+1;\ \  s=[\log_a N_0]
+ m .
\end{equation}
Define the coefficients $c_n$, $a_i$, $b_j$ and the function $P(x)$ in the following way:
\begin{equation}\label{2.14}
P(x)=\gamma\cdot \chi_{\Delta_m^{(k)}}(x)\cdot
I_{\nu_0}^{(1)}(a^sx),\ \ x\in [0,1]\ ,
\end{equation}
\begin{equation}\label{2.15}
c_n=c_n(P)=\int_0^1P(x)\overline{\psi_n}(x)dx\ ,\ \forall n\geq 0 ,
\end{equation}
\begin{equation}\label{2.16}
a_i=a_i(\chi_{\Delta_m^{(k)}})\ ,\ 0\leq i<a^m\ ,\ \
b_j=b_j(I_{\nu_0}^{(1)})\ ,\ 1\leq j< a^{\nu_0}\ .
\end{equation}
Taking into account \eqref{2.1}-\eqref{2.3}, \eqref{2.5}-\eqref{2.7}, \eqref{2.9}-\eqref{2.12} for $P(x)$ we obtain
\begin{equation}\label{2.17}
P(x)=\gamma\cdot \sum_{i=0}^{a^m - 1}a_i \psi_i(x)\cdot
\sum_{j=1}^{a^{\nu_0} - 1}b_j \psi_j(a^sx)=
\end{equation}
$$
=\gamma\cdot \sum_{j=1}^{a^{\nu_0 }- 1}b_j\cdot\sum_{i=0}^{a^m -
1}a_i \psi_{j\cdot a^s+i}(x) =\sum_{k=N_0}^ N c_k\psi_k(x)\ ,
$$
where
\begin{equation}\label{2.18}
c_k=c_k(P)=
\begin{cases} -\mathcal{K}\cdot {\gamma \over a^m}\ \hbox{or}\
0\ ,\ \hbox{if}\ k\in [N_0, N]\ \ \cr 0\ ,\qquad \hbox{if}\ k\notin
[N_0, N],
\end{cases}
\end{equation}
\begin{equation}\label{2.19}
\mathcal{K}\in \Omega_a,\ \ |\mathcal{K}|=1,\ \
N=a^{s+\nu_0}+a^m-a^s-1.
\end{equation}
Set
$$
E=\{x\in \Delta :P(x)=\gamma\}\ .
$$
By \eqref{2.7}, \eqref{2.8} and \eqref{2.14} we have
$$
|E|=a^{-m}(1-a^{-\nu_0})>(1-\epsilon)|\Delta|,
$$
$$
 P(x)=
 \begin{cases} \gamma \ ,\ \hbox{if}\ x\in E\ ,\cr
{\gamma(1-a^{\nu_0})}\ ,\ \hbox{if}\ x\in \Delta\setminus E\ ,\cr 0\
,\ \hbox{if}\ x\notin\Delta\ .
\end{cases}
$$
Hence and from \eqref{2.13} we get
$$
\int_0^1|P(x)|dx=2\cdot
|\gamma||\Delta|\cdot(1-a^{-\nu_0}),
$$
and taking into account that $a\geq 2$ we have
$$
{1\over 2}\cdot |\gamma| \cdot |\Delta|<\int_0^1|P(x)|dx<2\cdot
|\gamma| \cdot |\Delta|.
$$
From relations \eqref{2.13}, \eqref{2.18} and \eqref{2.19} we obtain
$$
\max_{N_0\leq m \leq N}\int_0^1\left|\sum_{k=N_0}^m
c_k\psi_k(x)\right|dx<\left[\int_0^1|P(x)|^2dx\right]^{1\over 2}\leq
\left[\sum_{k=N_0}^N c_k^2\right]^{1\over 2}=
$$
$$
=|\gamma| \cdot |\Delta|\cdot \sqrt{a^{\nu_0+s}+a^m}=|\gamma| \cdot
\sqrt{|\Delta|}\cdot \sqrt{a^{\nu_0}+1}<
$$
$$
<|\gamma| \cdot \sqrt{|\Delta|}\cdot \sqrt{a \over \varepsilon}<
a\cdot |\gamma| \cdot \sqrt{|\Delta|\over \varepsilon}.
$$

\end{proof}
\par\par\bigskip
\begin{lemma}\label{lem 2.2} For any given numbers $\ N_{0}>1$, $(N_0\in
\mathcal{N})$, $\varepsilon\in(0,1)$ and each function $f(x)\in L^1[0,1)$,
$||f||_1>0$ there exists a measurable set $E\subset [0,1)$, function
$g(x)\in L^1[0,1)$ and a polynomial by $\Psi_a$ system of the form
$$
P(x)=\sum_{k=N_{0}}^{N}c_{k}\psi _{n_k}(x),\ \ n_k\uparrow
$$
satisfying the following conditions:
$$
|E|>1-\varepsilon, \leqno 1)
$$
$$
 f(x)=g(x),\ {\hbox {��� ����}}\  \ x\in E, \leqno 2)
$$
$$
{1\over 2} \int_0^1|f(x)|dx< \int_0^1|g(x)|dx<
3\int_0^1|f(x)|dx.\leqno 3)
$$
$$
\int_0^1|P(x)-g(x)|dx<\varepsilon.\leqno 4)
$$
$$
\varepsilon>|c_k|\geq |c_{k+1}|>0.\leqno 5)
$$
$$
\max_{N_0\leq m \leq N}\int_0^1\left|\sum_{k=N_0}^m
c_k\psi_{n_k}(x)\right|dx<3\int_0^1|f(x)|dx.\leqno 6)
$$
\end{lemma}

\begin{proof}
 Consider the step function
\begin{equation}\label{2.20}
\varphi(x)=\sum_{\nu =1}^{\nu_0}\gamma _\nu\cdot \chi _{\Delta
_\nu}(x),
\end{equation}
where $\Delta _\nu$ are $a$-dyadic, not crosse intervals of the form $\Delta
_{m}^{(k)}=[{k-1\over a^m},{k\over a^m})$, $k=1,2,...,a^m$ so that
\begin{equation}\label{2.21}
0<|\gamma _\nu|^2|\Delta _\nu|<{\varepsilon^3 \over 16a^2}\cdot
\left(\int_0^1|f(x)|dx \right)^2.
\end{equation}
\begin{equation}\label{2.22}
0<|\gamma _{\nu_0}||\Delta _{\nu_0}|<...<|\gamma _\nu||\Delta
_\nu|<...<|\gamma _{1}||\Delta _{1}|<{\varepsilon \over 2}.
\end{equation}
\begin{equation}\label{2.23}
\int_0^1|f(x)-\varphi(x)|dx<\min \{{\varepsilon \over
4};{\varepsilon \over 4}\int_0^1|f(x)|dx \}.
\end{equation}
Applying Lemma 2.1 successively, we can find the sets
$E_\nu \subset [0,1)$ and a polynomial
\begin{equation}\label{2.24}
P_{\nu}(x)=\sum_{k=N_{\nu-1}}^{N_\nu-1}c_{k}\psi
_{n_k}(x),\ \ 1\leq \nu \leq \nu_0,
\end{equation}
which, for all $1\leq \nu \leq \nu_0$, satisfy the following conditions:
\begin{equation}\label{2.25}
|c_k|=|\gamma_{\nu}|\cdot |\Delta _{\nu}|,\ \ {\hbox {���}}\ \
k\in[N_{\nu-1},N_{\nu})
\end{equation}
\begin{equation}\label{2.26}
|E_{\nu}|>(1-\varepsilon )\cdot |\Delta _{\nu}|,
\end{equation}
\begin{equation}\label{2.27}
P_{\nu}(x)=
\begin{cases} \gamma _{\nu}\ :\quad \hbox{����}\ x\in \
E_{\nu}\cr 0\quad :\quad \hbox{����}\ x\notin \ \Delta _{\nu},
\end{cases}
\end{equation}
\begin{equation}\label{2.28}
{1\over 2}|\gamma_{\nu}|\cdot |\Delta
_{\nu}|<\int_0^1|P_{\nu}(x)|dx<2|\gamma_{\nu}|\cdot |\Delta _{\nu}|.
\end{equation}
\begin{equation}\label{2.29}
{\max_{N_{\nu-1}\leq m \leq N_{\nu}}\int_0^1\big|\sum_{k=N_0}^m
c_k\psi_{n_k}(x)\big|<a\cdot |\gamma_{\nu}}| \cdot
\sqrt{|\Delta_{\nu}|\over \varepsilon}.
\end{equation}
Define a set $E$, a function $g(x)$ and a polynomial $P(x)$ in the following away:
\begin{equation}\label{2.30}
P(x)=\sum_{\nu =1}^{\nu_0}P_{\nu}(x)=\sum_{k=N_{0}}^{N}c_{k}\psi
_{n_k}(x),\ \ N=N_{\nu_0}-1.
\end{equation}
\begin{equation}\label{2.31}
g(x)=P(x)+f(x)-\varphi(x).
\end{equation}
\begin{equation}\label{2.32}
E=\bigcup_{\nu=1}^{\nu_0}E_{\nu}.
\end{equation}

From \eqref{2.20},\eqref{2.23}, \eqref{2.26}-\eqref{2.28}, \eqref{2.30}-\eqref{2.32} we have

$$
|E|>1-\varepsilon \ ,
$$
$$
f(x)=g(x)\ ,\ \quad \hbox{ for}\ x\in E,
$$
$$
{1\over 2} \int_0^1|f(x)|dx< \int_0^1|g(x)|dx<
3\int_0^1|f(x)|dx.
$$
By \eqref{2.22}, \eqref{2.23}, \eqref{2.25} and \eqref{2.31} we get
$$
\int_0^1|P(x)-g(x)|dx=\int_0^1|f(x)-\varphi(x)|dx<\varepsilon.
$$
$$
\varepsilon>|c_k|\geq |c_{k+1}|>0,\ \ {\hbox {for}}\ \
k=N_0,N_0+1,...,N-1.
$$
That is, assertions 1)-5) of Lemma 2.2 actually hold. We now verify assertion 6). For any number  $m$, $N_0\leq m \leq N$
we can find $j$, $1\leq j \leq \nu_0$ such that $N_{j-1}<m\leq N_j$.
then by (2.24) and (2.30) we have
$$
\sum_{k=N_0}^m
c_k\psi_{n_k}(x)=\sum_{n=1}^{j-1}P_n(x)+\sum_{k=N_{j-1}}^m
c_k\psi_{n_k}(x).
$$
hence and from relations \eqref{2.21}, \eqref{2.23}, \eqref{2.28}, \eqref{2.29} we obtain
$$
\int_0^1 \left|\sum_{k=N_0}^m c_k\psi_{n_k}(x)\right|dx\leq
\sum_{\nu=1}^{\nu_0}\int_0^1|P_{\nu}(x)|dx+\int_0^1
\left|\sum_{k=N_{j-1}}^m c_k\psi_{n_k}(x)\right|dx<
$$
$$
<2\int_0^1|\varphi(x)|dx+a\cdot |\gamma_j|\cdot
\sqrt{|\Delta_j|\over
\varepsilon}<3\int_0^1|f(x)|dx.
$$
\end{proof}
\par\par\bigskip

\section{ Main results }
\par\par\bigskip

\begin {proof} Let
\begin{equation}\label{3.1}
\{ f_n(x)\}_{n=1}^\infty
\end{equation}
be a sequence of all step functions, values and constancy interval
endpoints of which are rational numbers. Applying Lemma 2.2
consecutively, we can find a sequences of functions
$\{ \overline{g}_n(x)\}$ of
sets $\{ E_n\}$ and a sequence of polynomials
\begin{equation}\label{3.2}
\overline{P}_n(x)=\sum_{k=N_{n-1}}^{N_n-1}c_{m_k}\psi _{m_k}(x),\ \
N_0=1,\ \ |c_{m_k}|>0
\end{equation}
which satisfy the conditions:
\begin{equation}\label{3.3}
|E_n|>1-\varepsilon\cdot 4^{-8(n+2)}
\end{equation}
\begin{equation}\label{3.4}
 f_n(x)=\overline{g}_n(x),\ {\hbox {for \ \ all}}\  \ x\in E_n,
\end{equation}
\begin{equation}\label{3.5}
{1\over 2} \int_0^1|f_n(x)|dx< \int_0^1|\overline{g}_n(x)|dx<
3\int_0^1|f_n(x)|dx.
\end{equation}
\begin{equation}\label{3.6}
\int_0^1|{\overline{P}}_n(x)-\overline{g}_n(x)|dx<4^{-8(n+2)}.
\end{equation}
\begin{equation}\label{3.7}
\max_{N_{n-1}\leq M \leq N_n}\int_0^1\left|\sum_{k=N_{n-1}}^M
c_{m_k}\psi_{m_k}(x)\right|dx<3\int_0^1|f_n(x)|dx.
\end{equation}
\begin{equation}\label{3.8}
{1\over n}>|c_{m_k}|> |c_{m_{k+1}}|>|c_{m_{N_n}}|>0.
\end{equation}
Set
\begin{equation}\label{3.9}
\sum_{k=1}^\infty c_{m_k}\psi _{m_k}(x)=\sum_{n=1}^\infty
\overline{P}_n(x)=
\sum_{n=1}^\infty\sum_{k=N_{n-1}}^{N_n-1}c_{m_k}\psi _{m_k}(x),
\end{equation}
 and
\begin{equation}\label{3.10}
E=\bigcap_{n=1}^\infty E_n.
\end{equation}
It is easy to see that (see \eqref{3.3}), $|E|>1-\varepsilon$.

Now we consider a series
$$
\sum_{i=1}^\infty c_i\psi _i(x)
$$
where $c_i=c_{m_k}$ ��� $i\in[m_k,m_{k+1})$. From \eqref{3.8} it follows that
$|c_i|\downarrow0$.

Let  given any function $f(x)\in L^1[0,1)$ then we can choose a subsequence
$\{ f_{s_n}(x)\}_{n=1}^\infty$ from \eqref{3.1} such that
\begin{equation}\label{3.11}
\lim_{N\to \infty}\int_0^1\left|\sum_{n=1}^Nf_{s_n}(x)-f(x)
\right|dx=0,
\end{equation}
\begin{equation}\label{3.12}
\int_0^1|f_{s_n}(x)|dx\leq \epsilon\cdot 4^{-8(n+2)}, n\geq 2,
\end{equation}
where
\begin{equation}\label{3.13}
\epsilon=\min\{ {\varepsilon\over 2},\int_E|f(x)|dx\}.
\end{equation}
We set
\begin{equation}\label{3.14}
g_1(x)=\overline{g}_{s_1}(x),\ \
P_1(x)=\overline{P}_{s_1}(x)=\sum_{k=N_{s_1-1}}^{N_{s_1}-1}c_{m_k}\psi
_{m_k}(x)
\end{equation}
It is easy to see that
$$
\int_0^1|f(x)-f_{k_1}(x)|<{\epsilon\over 2}
$$
Taking into account \eqref{3.5}, \eqref{3.7} and \eqref{3.14} we have
$$
\max_{N_{s_1-1}\leq M \leq
N_{s_1}}\int_0^1\left|\sum_{k=N_{s_1-1}}^M
c_{m_k}\psi_{m_k}(x)\right|dx<3\int_0^1|f_{s_1}(x)|dx<6\int_0^1|g_1(x)|dx.
$$
Then assume that numbers
$\nu_1,\nu_2,...,\nu_{q-1}$ ($\nu_1=s_1$), functions $g_n(x)$,
$f_{\nu_n}(x)$, $n=1,2,...,q-1$ and polynomials
$$
P_n(x)=\sum_{k=M_n}^{\overline{M}_n}c_{m_k}\psi _{m_k}(x),\ \
M_n=N_{\nu_n-1},\ \ {\overline{M}_n}=N_{\nu_n}-1,
$$
are chosen in such a way that the following condition is satisfied:
\begin{equation}\label{3.15}
g_n(x)=f_{s_n}(x),\ \ x\in E_{\nu_n},\ \ 1\leq n \leq q-1,
\end{equation}
\begin{equation}\label{3.16}
\int_0^1|g_n(x)|dx<4^{-3n}\epsilon,\ \ 1\leq n \leq q-1,
\end{equation}
\begin{equation}\label{3.17}
\int_0^1\left|\sum_{k=2}^n(P_k(x)-g_k(x))\right|dx<4^{-8(n+1)}\epsilon,\
\ 1\leq n \leq q-1,
\end{equation}
\begin{equation}\label{3.18}
\max_{M_n\leq M \leq \overline{M}_n}\int_0^1\left|\sum_{k=M_n}^M
c_{m_k}\psi_{m_k}(x)\right|dx<4^{-3n}\epsilon,\ \ 1\leq n \leq q-1.
\end{equation}
We choose a function $f_{\nu_q}(x)$ from the sequence (3.1) such
that
\begin{equation}\label{3.19}
\int_0^1\left|f_{\nu_q}(x)-\left[f_{s_q}(x)-\sum_{k=2}^n(P_k(x)-g_k(x))\right]\right|dx<4^{-8(q+2)}\epsilon.
\end{equation}
This with \eqref{3.11} imply
$$
\int_0^1\left|f_{\nu_q}(x)-\sum_{k=2}^n(P_k(x)-g_k(x))\right|dx<4^{-8q-1}\epsilon,
$$
and taking into account relation \eqref{3.19} we get
\begin{equation}\label{3.20}
\int_0^1|f_{\nu_q}(x)|dx<4^{-8q}\epsilon.
\end{equation}
We set
\begin{equation}\label{3.21}
P_q(x)=\overline{P}_{\nu_q}(x)=\sum_{k=M_q}^{\overline{M}_q}c_{m_k}\psi
_{m_k}(x),
\end{equation}
where
$$
 M_q=N_{\nu_q-1},\ \ {\overline{M}_q}=N_{\nu_q}-1,
$$
\begin{equation}\label{3.22}
g_q(x)=f_{s_q}(x)+[\overline{g}_{\nu_q}(x)-f_{\nu_q}(x)]
\end{equation}
By \eqref{3.4}-\eqref{3.7}, \eqref{3.17}-\eqref{3.22} we have
\begin{equation}\label{3.23}
g_q(x)=f_{s_q}(x),\ \ x\in E_{\nu_q},
\end{equation}
\begin{equation}\label{3.24}
\int_0^1|g_q(x)|dx\leq
\end{equation}
$$
\leq\int_0^1\left|f_{\nu_q}(x)-\left[f_{s_q}(x)-\sum_{k=2}^n(P_k(x)-g_k(x))\right]\right|dx+
$$
$$
+\int_0^1|\overline{g}_{\nu_q}(x)|dx+\int_0^1\left|\sum_{k=2}^n(P_k(x)-g_k(x))\right|dx<4^{-3n}\epsilon,
$$
\begin{equation}\label{3.25}
\int_0^1\left|\sum_{k=2}^q(P_k(x)-g_k(x))\right|dx\leq
\end{equation}
$$
\leq\int_0^1\left|f_{\nu_q}(x)-\left[f_{s_q}(x)-\sum_{k=2}^n(P_k(x)-g_k(x))\right]\right|dx+
$$
$$
+\int_0^1|\overline{P}_{\nu_q}(x)-\overline{g}_{\nu_q}(x))|dx<4^{-8(n+1)}\epsilon,
$$
\begin{equation}\label{3.26}
\max_{M_q\leq M \leq \overline{M}_q}\int_0^1\left|\sum_{k=M_q}^M
c_{m_k}\psi_{m_k}(x)\right|dx\leq
3\int_0^1|f_{\nu_q}(x)|dx<4^{-3n}\epsilon.
\end{equation}

Thus, by induction we can choose the sequences of sets
$\{E_q\}$, functions $\{g_q(x)\}$ and polynomials $\{P_q(x)\}$
such that conditions \eqref{3.23} - \eqref{3.26} are satisfied for all $q\geq
1.$
Define a function $g(x)$ and a series in the following away:
\begin{equation}\label{3.27}
g(x)=\sum_{n=1}^\infty g_n(x),
\end{equation}
\begin{equation}\label{3.28}
\sum_{n=1}^\infty
\delta_nc_n\psi_n(x)=\sum_{n=1}^\infty\left[\sum_{k=M_n}^{\overline{M}_n}c_{m_k}\psi
_{m_k}(x) \right],
\end{equation}
where
$$
\delta_n=
\begin{cases}1\ ,\ \hbox{if}\ i=m_k,\ \ \hbox{where}\ \ k\in \displaystyle{\bigcup_{q=1}^\infty[M_q,\overline{M}_q]} \cr \\
 \cr 0 ,\ \ \hbox{in the other case }. \cr
\end{cases}
$$
Hence and from relations \eqref{3.5}, \eqref{3.10}, \eqref{3.15}, \eqref{3.27} �������, ���
\begin{equation}\label{3.29}
g(x)=f(x),\ \ \hbox{���}\ \ x\in E,\ \ g(x)\in L^1[0,1),
\end{equation}
\begin{equation}\label{3.30}
{1\over 2} \int_0^1|f(x)|dx< \int_0^1|g(x)|dx< 4\int_0^1|f(x)|dx.
\end{equation}
Taking into account \eqref{3.21}, \eqref{3.24}-\eqref{3.28} we obtain that the series \eqref{3.28} convergence
to $g(x)$ in $L^1[0,1)$ metric and consequently is its Fourier series by $\Psi_a$ system, $a\geq 2$.

From Definition 1.3, and from relations \eqref{3.13}, \eqref{3.18}, \eqref{3.30} for any natural number $m$ there is
$N_m$ so that
$$
||G_m(g)||_1=||S_m(g)||_1=\int_0^1\left|\sum_{n=1}^\infty
\delta_nc_n\psi_n(x)\right|dx\leq4\int_0^1|f(x)|dx
$$
$$
\leq \sum_{n=1}^\infty\left( \max_{M_n\leq M \leq
\overline{M}_n}\int_0^1\left|\sum_{k=M_n}^M
c_{m_k}\psi_{m_k}(x)\right|dx\right)\leq
$$
$$
\leq 2\int_0^1|g_1(x)|dx+\epsilon\cdot\sum_{n=2}^\infty4^{-n}\leq
$$
$$
\leq 3\int_0^1|g(x)|dx\leq 12\int_0^1|f(x)|dx=12||f||_1.
$$

 \end {proof}
\par\par\bigskip
{\bf Acknowledgement.} The author thanks Professor M.G.Grigorian for his attention to
this paper.

\bibliographystyle{amsplain}

\end{document}